\documentclass[a4paper,12pt]{amsart}

\usepackage[utf8]{inputenc}
\usepackage[english]{babel}
\usepackage{amssymb}
\usepackage[alphabetic]{amsrefs}

\newtheorem{theorem}{Theorem}[section]
\newtheorem{lemma}[theorem]{Lemma}
\newtheorem{proposition}[theorem]{Proposition}
\newtheorem{cor}[theorem]{Corollary}
\theoremstyle{definition}
\newtheorem{definition}[theorem]{Definition}

\theoremstyle{remark}
\newtheorem{remark}[theorem]{Remark}

\newcommand{\CC}{\mathbb{C}}
\newcommand{\NN}{\mathbb{N}}

\newcommand{\id}{\mathop{\mathrm{id}}}

\newcommand{\cont}{\mathcal{C}}
\newcommand{\hol}{{\mathcal{O}_\mathrm{hol}}}
\newcommand{\regular}{{\mathcal{O}_\mathrm{alg}}}
\newcommand{\aut}{\mathop{\mathrm{Aut}_\mathrm{hol}}}
\newcommand{\autalg}{\mathop{\mathrm{Aut}_\mathrm{alg}}}

\title[The density property for Gizatullin surfaces]{The density property for Gizatullin surfaces with reduced degenerate fibre}
\date{21. November 2016}

\author[R.~Andrist]{Rafael~B.~Andrist}
\address{Rafael B. Andrist \\ School of Mathematics and Natural Sciences \\ University of Wuppertal \\ Germany}
\email{rafael.andrist@math.uni-wuppertal.de}


\begin{document}

\begin{abstract}
All Gizatullin surfaces that admit such a $\mathbb{C}^+$-action for which the quotient is a $\mathbb{C}^1$-fibration with a reduced degenerate fibre, have the density property. We also give a description of the identity component of the group of holomorphic automorphisms of these surfaces.
\end{abstract}

\maketitle

\section{Introduction}

In this article we want to contribute to the classification of smooth complex affine-algebraic surfaces with the density property by showing that a large subclass of so-called Gizatullin surfaces has the density property.

\smallskip

The \emph{density property} has been introduced by Varolin \cites{Varolin1, Varolin2} to describe precisely what it means for a complex manifold to have ``large'' group of holomorphic automorphisms. Let us recall the following definitions:

\begin{definition}[\cite{Varolin1}]
A complex manifold $X$ has the \emph{density property} if the Lie algebra $\mathrm{Lie}_{\mathrm{hol}}(X)$ generated by $\CC$-complete holomorphic vector fields on $X$ is dense in the Lie algebra of all holomorphic vector fields $\mathrm{VF}_{\mathrm{hol}}(X)$ on $X$ w.r.t.\ compact-open topology.
\end{definition}

\begin{definition}[\cite{Varolin1}]
An algebraic manifold $X$ has the \emph{algebraic density property} if the Lie algebra $\mathrm{Lie}_{\mathrm{alg}}(X)$ generated by $\CC$-complete algebraic vector fields on $X$ coincides with the Lie algebra of all algebraic vector fields $\mathrm{VF}_{\mathrm{alg}}(X)$ on $X$.
\end{definition}

The algebraic density property implies the density property.
The classical and most important example of a complex manifold with the (algebraic) density property is $\CC^n$ for $n \geq 2$ where the (algebraic) density property was established by Anders\'en and Lempert \cite{AndersenLempert}. It turned out that there are much more examples of complex affine-algebraic or Stein manifolds with the density property.

However, in complex dimension $2$, only very few examples are known and even in complex dimension $2$, a classification of all such surfaces is not in sight. In particular, also the description of all Stein surfaces $X$ on which the group of holomorphic automorphisms $\aut(X)$ acts transitively, is completely unclear. But in the algebraic category, the question of algebraic transitivity was ``almost'' resolved in the papers of Gizatullin and Danilov \cite{Gi}, \cite{dan-giz-autos}. We will need the following definition:

\begin{definition} A normal complex algebraic surface $X$ is called \emph{quasi-homogeneous} if the natural action of its group of algebraic automorphisms $\autalg(X)$ has an open orbit whose complement is at most finite.
\end{definition}

With the exception of the two-dimensional torus $\CC^\ast \times \CC^\ast$  and of $\CC \times \CC^\ast$, every affine quasi-homogeneous surface admits a completion $\bar X$ by a simple normal crossing divisor such that the dual graph $\Gamma$ of its boundary $\bar X \setminus X$ is a linear rational graph \cites{Gi, dan-giz-autos} which can be always chosen in the following standard form (a so-called \emph{standard zigzag})
\[
[[0, 0,-r_2 , \dots, -r_d]]
\]
where $d \geq 2$ and $r_j \geq 2$ for $j = 2, \dots, d$. The numbers given in $[[0, 0,-r_2 , \dots, -r_d]]$ are the self-intersection numbers of the the involved projective lines. This leads to the following definition.

\begin{definition}
A normal complex affine-algebraic surface that admits a completion by a simple normal crossing divisor such that its dual graph is a linear rational graph, is called a \emph{Gizatullin surface}.
\end{definition}

The following results are known so far:
For $d = 2$, resp.\ a zigzag  $[[0, 0, -r_2]]$, one obtains the so-called Danielewski surfaces. The density property for Danielewski surfaces was proved by Kaliman and Kutzschebauch \cite{KKhyper}.
For so-called Danilov--Gizatullin surfaces, which are described by a zigzag $[[0, 0, -2, \dots, -2]]$ and are (as smooth Gizatullin surfaces) completely determined up to isomorphism by the length of the zigzag, the density property was proved by Donzelli \cite{dan-giz-dens}.
For $d = 3$, the case of Gizatullin surfaces described by a zigzag $[[0, 0, -r_2, -r_3]]$, the density property was established recently by Kutzschebauch and Poloni and the author \cite{lengthfour}.

\bigskip
A Gizatullin surface always admits two non-conjugate $\CC^+$-actions, see e.g.\ Flenner, Kaliman and Zaidenberg \cite{wgraphs}.
In this article we want to consider Gizatullin surfaces with a priori any possible standard zigzag as boundary divisor, but we make a restriction regarding the $\CC^+$-actions. We require that there exists a $\CC^+$-action on the surface such that its quotient is a $\CC^1$-fibration with a single reduced degenerate fibre. One can always find a fibration with a single degenerate fibre, but in general this fibre will not be reduced.

Embeddings of precisely such surfaces, sometimes also called generalized Danielewski surfaces, have been constructed by Dubouloz \cite{Dubouloz}. In the next section we cite the defining equations \eqref{eqdef1} and \eqref{eqdef2} that were derived by Duboloz.

\begin{theorem}
\label{thmdensprop}
Let $S$ be a smooth Gizatullin surface that is described by the equations \eqref{eqdef1} and \eqref{eqdef2}. Then $S$ has the density property.
\end{theorem}

\begin{cor}
Let $S$ be a smooth Gizatullin surface with a $\CC^+$-action such that its quotient is a $\CC^1$-fibration with a single reduced degenerate fibre. Then $S$ has the density property.
\end{cor}

In particular, these equations comprise all classes of smooth Gizatullin surfaces for which the density property has been established so far, but contain also many new examples of Gizatullin surfaces of any possible boundary divisor for which the density property was not known. Moreover, these are also new examples of complex manifolds that are elliptic in the sense of Gromov, see Section \ref{secgeom} for more details.

We also describe the identity component of the group of holomorphic automorphisms of these surfaces in Theorem \ref{thmautos}.

\begin{definition}\cite{singular}*{Def.~1.1}
Let $X$ be a complex affine-algebraic variety. We say that $X$ has the \emph{algebraic density property relative to a subvariety} $A \subset X$ if $\mathrm{Lie}_{\mathrm{alg}}(X, A)$, the Lie algebra generated by the complete algebraic vector fields on $X$ that vanish in $A$ and $\mathrm{VF}_{\mathrm{alg}}(X, A)$, the Lie algebra of all algebraic vector fields on $X$ that vanish on $A$ satisfy
\[
\mathcal{I}_A^\ell \cdot \mathrm{VF}_{\mathrm{alg}}(X, A) \subseteq \mathrm{Lie}_{\mathrm{alg}}(X, A)
\]
for some $\ell \in \NN_0$, where $\mathcal{I}_A$ denotes the vanishing ideal of $A$.
\end{definition}

\begin{theorem}
\label{thmrelativedensprop}
Let $S$ be a Gizatullin surface that is described by the equations \eqref{eqdef1} and \eqref{eqdef2}. Then $S$ has the relative algebraic density property with respect to a finite set $A \subset S$.
\end{theorem}

It is an open question whether we can choose $A = \emptyset$ for a smooth Gizatullin surface $S$ described by the equations \eqref{eqdef1} and \eqref{eqdef2}. Counter-examples of smooth Gizatullin surfaces where the action of $\autalg$ is not transitive are known due to Kovalenko \cite{nontrans}, but do not fall into this class.

\smallskip
Given Remark \ref{remccc} regarding the \emph{charming chain of charts} we conjecture that any Gizatullin surface (note that this excludes $\CC^\ast \times \CC^\ast$) that admits a transitive action of its group of algebraic automorphisms has the algebraic density property.

In this context we also mention that a larger class of complex affine-algebraic surfaces,  called \emph{generalized Gizatullin surfaces},  has been introduced by Kaliman, Kutzschebauch and Leuenberger \cite{KaKuLeu}. They could also provide potential candidates with the algebraic density property although no such examples with the algebraic density property are known.

\bigskip
The article is organized as follows: In the next section we recall the defining equations. Using those equations, we find large affine charts that cover the surfaces  and $\CC$-complete algebraic vector fields and their description in these charts in Section \ref{seccharts}. As a first step towards the density property, we show that the group of holomorphic automorphisms acts transitively in Section \ref{sectrans}. We can also derive directly the so-called holomorphic flexibility (see there). After these preparations we can finally prove the density property. The main ingredients are the calculations in Section \ref{secsubmod}. In Section \ref{secdensity} we provide the theoretical background and in Section \ref{secgeom} we discuss some geometric consequences.

\section{Defining Equations}

According to [Dub04, Section 5.2], all generalized Danielewski surfaces with a boundary divisor of length $d+1$ can be described by the following equations in $\CC[z_0, z_1, \dots z_d, z_{d+1}]$
\begin{align}
\label{eqdef1}
z_0 z_{i + 1} &= \left( \prod_{k=1}^{i-1} \tilde{P}_k(z_k) \right) \cdot P_i(z_i) & 1 \leq i \leq d \\
\label{eqdef2}
(z_{j-1} - \lambda_{j-1}) z_{i+1} &= z_j \cdot \left( \prod_{k=j}^{i-1} \tilde{P}_k(z_k) \right) \cdot P_i(z_i) & 2 \leq j \leq i \leq d
\end{align}
where we use the convention that an empty product is equal to $1$. 
The functions $P_i, \tilde{P}_i \colon \CC \to \CC$ are polynomials in one variable and satisfy the following relation:
\begin{align}
\label{defpolyrel}
P_i(z_i) = (z_i - \lambda_{i}) \cdot \tilde{P}_i(z_i)
\end{align}
We will call $\lambda_i$ the \emph{distinguished root} of $P_i$.
These polynomials $\tilde{P}_i$ are required to have simple roots. For our calculations, this will however not matter and we will consider all possible surfaces described by these equations.

\smallskip
We may rewrite the equations \eqref{eqdef1} and \eqref{eqdef2} as follows:
\begin{align}
\label{eqdef3}
z_0 z_{i + 1} &= (z_i - \lambda_i) \cdot \left( \prod_{k=1}^{i} \tilde{P}_k(z_k) \right) & 1 \leq i \leq d \\
\label{eqdef4}
(z_{j-1} - \lambda_{j-1}) z_{i+1} &= z_j \cdot (z_i - \lambda_i) \cdot \left( \prod_{k=j}^{i} \tilde{P}_k(z_k) \right) & 2 \leq j \leq i \leq d
\end{align}
Division of \eqref{eqdef3} by \eqref{eqdef4} and multiplication by $z_j (z_{j-1} - \lambda_j)$ yields
\begin{equation}
\label{eqdef5}
 z_0 z_j  = (z_{j-1} - \lambda_j) \cdot \left( \prod_{k=1}^{j-1} \tilde{P}_k(z_k) \right) = (z_{j-1} - \lambda_j) \cdot \left( \prod_{k=1}^{j-1} \hat{P}_k(z_k - \lambda_k) \right) 
\end{equation}
Whether we consider the polynomials $\tilde{P}_k$ as functions of $z_k$ or $z_k - \lambda_k$ does of course not matter. We thus obtain the set of equations given in  \cite{Dubouloz}*{Theorem 4.3} after appropriate shifts and renumbering of coordinates. For our purpose, it will be more convenient to work with the equations \eqref{eqdef1} and \eqref{eqdef2}.
\newpage
\section{Large affine charts}
\label{seccharts}

It will turn out to be convenient to work with the following, large affine charts. Their definition emerges directly from the defining equations \eqref{eqdef1} and \eqref{eqdef2}.

\smallskip

Let $\varphi_0 \colon \CC_{z_0}^\ast \times \CC_{z_1} \to S$ be the local parametrization given by
\begin{align*}
z_2 &= \frac{P_1(z_1)}{z_0} \\
z_3 &= \frac{P_1(z_1)}{z_1 - \lambda_1} \frac{P_2(z_2)}{z_0} = \tilde{P}_1(z_1) \frac{P_2(z_2)}{z_0} \\
\vdots \quad &= \quad \vdots \\
z_{d+1} &= \frac{P_1(z_1)}{z_1 - \lambda_1} \frac{P_2(z_2)}{z_2 - \lambda_2} \cdots \frac{P_{d-1}(z_{d-1})}{z_{d-1} - \lambda_{d-1}} \frac{P_d(z_d)}{z_0} \\
&= \tilde{P}_1(z_1) \tilde{P}_2(z_2) \cdots \tilde{P}_{d-1}(z_{d-1}) \frac{P_d(z_d)}{z_0}
\end{align*} 
The equations are to be understood s.t.\ all $z_2, \dots, z_{d+1}$ are replaced inductively by expressions in $z_0$ and $z_1$. The important observation is that the other coordinates of points in $S$ are completely determined by $(z_0, z_1) \in \CC^\ast \times \CC$ if $z_0 \neq 0$.

%
%


\smallskip

For each $i=1, \dots, d$, let $\varphi_i \colon \CC_{z_i-\lambda_i}^\ast \times \CC_{z_{i+1}}^\ast \to S$ be the local parametrization given by
\begin{align*}
z_0 &= \frac{\tilde{P}_1(z_1) \tilde{P}_2(z_2) \cdots \tilde{P}_{i-1}(z_{i-1})  P_i(z_i)}{z_{i+1}} \\
z_1 &= \lambda_1 + \frac{z_2 \tilde{P}_2(z_2) \tilde{P}_{i-1}(z_{i-1}) P_i(z_i)}{z_{i+1}}\\
z_2 &= \lambda_2 + \frac{z_3 \tilde{P}_3(z_3) \tilde{P}_{i-1}(z_{i-1}) P_i(z_i)}{z_{i+1}}\\
\vdots &= \vdots \\
z_{i-1} &= \lambda_{i-1} + \frac{z_i P_i(z_i)}{z_{i+1}}\\
z_{i+2} &= \frac{z_{i+1} P_{i+1}(z_{i+1})}{z_i - \lambda_i} \\
\vdots \quad &= \quad \vdots \\
z_{d+1} &= \tilde{P}_{i+1}(z_{i+1}) \cdots \tilde{P}_{d-1}(z_{d-1}) \frac{z_{i+1} P_d(z_d)}{z_i - \lambda_i} \\
\end{align*}

\begin{remark}
The finitely many affine charts $\varphi_0(\CC^\ast \times \CC), \; \varphi_1(\CC^\ast \times \CC^\ast), \dots, \varphi_{d}(\CC^\ast \times \CC^\ast)$ are Zariski-open, hence dense, sets in $S$. The inverse maps are given by projections to the corresponding coordinate plane, $\varphi_i^{-1}(z_0, z_1, \dots, z_d, z_{d+1}) = (z_i, z_{i+1})$.
\end{remark}

\begin{lemma}
The surface $S$ is covered by the charts $\varphi_0, \varphi_1, \dots, \varphi_d$ if at least one of the distinguished roots $\lambda_1, \dots, \lambda_d$ is non-zero.

Otherwise, if all of them vanish, the zero point in $\CC_{z_0, \dots, z_{d+1}}^{d+2}$ is the only point that is not contained in any of these charts.
\end{lemma}
\begin{proof}
Assume that these charts do not cover $S$, then by the chart $\varphi_0$ we have $z_0 = 0$ and by each other chart $\varphi_i$, $i=1, \dots, d$, we have $z_i = \lambda_i$ and $z_{i+1} = 0$. Therefore, this case can only occur if $z_0 = 0, \dots, z_{d+1} = 0$ and $\lambda_1 = 0, \dots, \lambda_d = 0$.
\end{proof}

\begin{lemma}
If the surface $S$ is smooth, then at least one of the distinguished roots $\lambda_1, \dots, \lambda_d$ must be non-zero. Otherwise, the zero point is the only singularity of $S$.
\end{lemma}
\begin{proof}
Assume to get a contradiction that all distinguished roots vanish. By the preceeding lemma, we already know that the surface is smooth outside the zero point.

We consider the derivatives of the defining equations and observe that in this case, all derivatives of \eqref{eqdef2} vanish in the zero point, since $P_i(z_i) = z_i \cdot \tilde{P}_i(z_i)$ on the right-hand side. Similarly, all but one derivatives of  \eqref{eqdef3} vanish in the zero point. The only potentially (if $P_i$ has simple roots) non-vanishing derivative is the $\frac{\partial}{\partial z_i}$-derivative. However, this could only guarantee smoothness if the co-dimension of $S$ in $\CC^{d+2}$ is $1$, i.e.\ if $d=0$, and such a Gizatullin surface does not exist.
\end{proof}

\begin{lemma}
\label{lemzerochart}
For all $j \in \NN_0$
\begin{equation}
\label{eqcompletezerochart1}
{\varphi_0}_\ast \left( z_0^{j+d} \cdot \frac{\partial}{\partial z_1} \right)
\end{equation}
and
\begin{equation}
\label{eqcompletezerochart2}
{\varphi_0}_\ast \left( z_0^{j+d-1} \cdot (z_1 - \lambda_1) \cdot \frac{\partial}{\partial z_1} \right)
\end{equation}
and
\begin{equation}
\label{eqcompletezerochart3}
{\varphi_0}_\ast \left( (z_1-\lambda_1)^{j+d-1} \cdot z_0 \cdot \frac{\partial}{\partial z_0} \right)
\end{equation}
extend to complete holomorphic vector fields on $S$.
\end{lemma}
\begin{proof}
The completeness of these vector fields in the chart is clear, since the flow maps are given by
\begin{align*}
((z_1, z_2), t) &\mapsto (z_0, z_1 + t z_0^{j+d}) \\
((z_1, z_2), t) &\mapsto (z_0, \exp(t z_0^{j+d-1} ) \cdot (z_1 - \lambda_1) + \lambda_1) \\
((z_1, z_2), t) &\mapsto (\exp(t (z_1-\lambda_1)^{j+d-1}) \cdot z_0, z_1)
\end{align*}
If they extend holomorphically to the whole surface, then they extend as complete  vector fields, because they vanish in the complement of the chart.
We calculate the following derivatives:
\begin{align*}
\frac{\partial z_2}{\partial z_1} &= \frac{P^\prime(z_1)}{z_0} \\
\frac{\partial z_k}{\partial z_1} &= \frac{1}{z_0} \Bigg( \sum_{\ell=1}^{k-2} \tilde{P}_1(z_1) \cdots \widehat{\tilde{P}_{\ell}(z_{\ell})} \cdots \tilde{P}_{k-2}(z_{k-2}) P_{k-1}(z_{k-1}) \cdot \tilde{P}_\ell^\prime(z_\ell) \cdot \frac{\partial z_\ell}{\partial z_1} \\
&\quad + \tilde{P}_1(z_1) \cdots \tilde{P}_{k-2}(z_{k-2}) \cdot P_{k-1}^\prime(z_{k-1}) \cdot \frac{\partial z_{k-1}}{\partial z_1} \Bigg) 
\end{align*}
By induction in $k = 2, \dots, d+1$ we see that with each step, we raise the power in $\frac{1}{z_0}$ by one. We may however, e.g.\ for $z_3$, reduce the power $\frac{1}{z_0}$ in exchange for a power in $\frac{1}{z_1 - \lambda_1}$: 
\begin{align*}
\frac{\partial z_3}{\partial z_1} &= \frac{1}{z_0} \Bigg( \tilde{P}_1^\prime(z_1) \cdot P_{2}(z_{2}) + \tilde{P}_1(z_1) \cdot P_{2}^\prime(z_{2}) \cdot \frac{P_1^\prime(z_1)}{z_0} \Bigg) \\
 &= \frac{1}{z_0} \Bigg( \tilde{P}_1^\prime(z_1) \cdot P_{2}(z_{2}) + \frac{z_2}{z_1 - \lambda_1} \cdot P_{2}^\prime(z_{2}) \cdot P_1^\prime(z_1) \Bigg)
\end{align*}
This proves the holomorphic extension of \eqref{eqcompletezerochart1} and \eqref{eqcompletezerochart2}.
 
Next, we calculate the following derivatives:
\begin{align*}
\frac{\partial z_2}{\partial z_0} &= -\frac{P(z_1)}{z_0^2} = -\frac{z_2}{z_0} \\
\frac{\partial z_k}{\partial z_0} &= -\frac{z_k}{z_0} \\ &\quad+ \frac{\tilde{P}_1(z_1)}{z_0} \sum_{\ell=2}^{k-2} \tilde{P}_2(z_2) \cdots \widehat{\tilde{P}_{\ell}(z_{\ell})} \cdots P_{k-1}(z_{k-1}) \cdot \tilde{P}_\ell^\prime(z_\ell) \cdot \frac{\partial z_\ell}{\partial z_0} \\
&\quad + \frac{\tilde{P}_1(z_1)}{z_0} \tilde{P}_2(z_2) \cdots \tilde{P}_{k-2}(z_{k-2}) \cdot P_{k-1}^\prime(z_{k-1}) \cdot \frac{\partial z_{k-1}}{\partial z_0} 
\end{align*}
Now, by induction in $k = 2, \dots, d+1$ we see that with each step, we again raise the power in $\frac{1}{z_0}$ by one.

However, for $k \geq 3$, after multiplication by $(z_1 - \lambda_1)$ we obtain:
\begin{align*}
(z_1 - \lambda_1) \frac{\partial z_k}{\partial z_0} &= -(z_1 - \lambda_1) \frac{z_k}{z_0} \\ &\quad + z_2 \cdot \sum_{\ell=2}^{k-2} \tilde{P}_2(z_2) \cdots \widehat{\tilde{P}_{\ell}(z_{\ell})} \cdots P_{k-1}(z_{k-1}) \cdot \tilde{P}_\ell^\prime(z_\ell) \cdot \frac{\partial z_\ell}{\partial z_0} \\
&\quad + z_2 \tilde{P}_2(z_2) \cdots \tilde{P}_{k-2}(z_{k-2}) \cdot P_{k-1}^\prime(z_{k-1}) \cdot \frac{\partial z_{k-1}}{\partial z_0}
\end{align*}
This proves the holomorphic extension of \eqref{eqcompletezerochart3}.
\end{proof}

\begin{lemma}
Let $i = 1, \dots, d$. Then for all $j \in \NN$
\begin{equation}
\label{eqcompleteotherchart1}
{\varphi_i}_\ast \left( z_{i+1}^{j+d-1} \cdot (z_{i} - \lambda_{i}) \cdot \frac{\partial}{\partial z_{i}} \right)
\end{equation}
and
\begin{equation}
\label{eqcompleteotherchart2}
{\varphi_i}_\ast \left( (z_{i}-\lambda_{i})^{j+d-i-1} z_i^i \cdot z_{i+1} \cdot \frac{\partial}{\partial z_{i+1}} \right)
\end{equation}
extend to complete holomorphic vector fields on $S$.
\end{lemma}
\begin{proof}
We proceed as the proof of Lemma \ref{lemzerochart}. The proof is completely analogous for $z_k$ depending on $(z_{i}, z_{i+1})$ if $k \geq i+2$. We only need to consider the case $k \leq i-1$.
\begin{align*}
\frac{\partial z_{i-1}}{\partial z_{i}} &= -\frac{P(z_i) + z_i P_i^\prime(z_i)}{z_{i+1}} \\
\frac{\partial z_{i-1}}{\partial z_{i+1}} &= -\frac{z_{i-1} - \lambda_{i-1}}{z_{i+1}} \\
\end{align*}
For $1 \leq k \leq i-2$ we obtain
\begin{align*}
\frac{\partial z_{k}}{\partial z_{i}} &= \frac{1}{z_{i+1}} \frac{\partial}{\partial z_{i}} \left( z_{k+1} \tilde{P}_{k+1}(z_{k+1}) \dots \tilde{P}_{i-1}(z_{i-1}) \cdot P_i(z_i) \right) \\
\frac{\partial z_{k}}{\partial z_{i+1}} &= -\frac{z_{k}-\lambda_{k}}{z_{i+1}} \\
 &\quad + \frac{P_i(z_i)}{z_{i+1}} \frac{\partial}{\partial z_{i+1}} \left( z_{k+1} \tilde{P}_{k+1}(z_{k+1}) \dots \tilde{P}_{i-1}(z_{i-1}) \right) 
\end{align*}
The case $k=0$ is completely analogous except that the factor $z_{k+1}$ does not appear on the right-hand side.

Now, by induction in $k = i-1, i-2, \dots, 0$ we see that with each step, we raise the power in $\frac{1}{z_{i+1}}$ by one.

We can multiply by $z_{i}$ and hence exchange the increase of one power in $\frac{1}{z_{i+1}}$ by one power in $\frac{1}{z_{i}}$:
\begin{align*}
z_{i} \frac{\partial z_{k}}{\partial z_{i+1}} &= -z_{i} \frac{z_{k}-\lambda_{k}}{z_{i+1}} \\
 &\quad + (z_i - \lambda_i) \frac{\partial}{\partial z_{i+1}} \left( z_{k+1} \tilde{P}_{k+1}(z_{k+1}) \dots \tilde{P}_{i-1}(z_{i-1}) \right) \qedhere
\end{align*}
\end{proof}


\section{Holomorphic transitivity}
\label{sectrans}
\begin{proposition}
\label{proptrans}
Assume that at least one of the distinguished roots $\lambda_1, \dots, \lambda_d$ is non-zero.

The complete vector fields given (for $j=0$) in \eqref{eqcompletezerochart1}, \eqref{eqcompletezerochart2} and their pull-backs by certain flows of \eqref{eqcompletezerochart1}, and the complete vector fields given (for $j=0$) in \eqref{eqcompleteotherchart1} and \eqref{eqcompleteotherchart2} together span the tangent space of $S$ in every point.
\end{proposition}

\begin{proof}
The two vector fields given by \eqref{eqcompletezerochart1} and  \eqref{eqcompletezerochart3} for $j=0$ span the tangent space in every point of the chart $\varphi_0(\CC^\ast_{z_0} \times \CC_{z_1})$ for $z_1 \neq \lambda_1$:
\[
z_0^{d} \frac{\partial}{\partial z_0}
\text{ and }
z_0^{d-1} (z_1 - \lambda_1) \frac{\partial}{\partial z_1}
\]
For a point $(z_0, z_1)$ with $z_1 = \lambda_1$ we just pull back these vector fields by the flow of $z_0^{d} \frac{\partial}{\partial z_0}$.
The two vector fields given by \eqref{eqcompleteotherchart1} and  \eqref{eqcompleteotherchart2} for $j=0$ span the tangent space in every point of the chart $\varphi_i(\CC^\ast_{z_i-\lambda_i} \times \CC^\ast_{z_{i+1}})$.
\end{proof}

\begin{cor}
If $S$ is smooth, then the holomorphic automorphisms act transitively on $S$. If $S$ is singular, then $0$ is the only singularity and the holomorphic automorphisms act transitively on $S \setminus\{0\}$.
\end{cor}

But actually, the preceding proposition implies more, in fact it says that $S$ is \emph{holomorphically flexible}, i.e.\ there exist finitely many complete holomorphic vector fields that span the tangent space in every point (see Arzhantsev et al.\ \cite{A-Z}*{Def.~A.4}) and we have the following immediate corollaries, see also \cite{A-Z} for details:

\begin{definition}\cite{Forstneric-book}*{Chap.\ 5} 
A \emph{spray} on a complex manifold $X$ is a triple $(E, \pi, s)$ consisting of a holomorphic vector bundle $\pi \colon E \to X$ and a holomorphic map $s \colon E \to X$  such that for each point $x \in X$ we have $s(0_x) = x$ where $0_x$ denotes the zero in the fibre over $x$.

The spray $(E, \pi, s)$ is said to be \emph{dominating} if for every point $x \in X$ we have
\[
\mathrm{d}_{0_x} s (E_{x}) = \mathrm{T}_x X
\]
A complex manifold is called \emph{elliptic} if it admits a dominating spray.
\end{definition}

In this definition we adapted the convention used e.g.\ in the textbook \cites{Forstneric-book} identifying the fibre $E_x$ over $x$ with its tangent space in $0_x$.

\begin{cor}
If $S$ is smooth, then it is elliptic in the sense of Gromov, i.e.\ it admits a dominating spray.
\end{cor}
\begin{proof}
If we denote by $\varphi_1, \dots, \varphi_N \colon \CC \times S \to S$, $N = 2(d+2)$, the flows of the complete holomorphic vector fields that span the tangent spaces in every point, then the map
\[
\CC^N \times S \to S, \; ((t_1, \dots, t_N), z) \mapsto \varphi_{N,t_N} \circ \dots \circ \varphi_{1, t_1}(z)) 
\]
is a dominating spray.
\end{proof}

\begin{cor}
If $S$ is smooth, then the group of holomorphic automorphisms acts $m$-transitively on $S$ for any $m \in \NN$.
\end{cor}

\begin{proof}
This is a clever application of the implicit function theorem, see Varolin \cite{Varolin2} for details.
\end{proof}

\section{Generating a submodule}
\label{secsubmod}

We will see in the next section that we actually do not need to approximate all holomorphic vector fields by Lie combinations of complete holomorphic vector fields. It is sufficient to find a $\hol(S)$-submodule inside the Lie algebra generated by the complete holomorphic vector fields. This is the goal of this section.

\begin{remark}
The \emph{Kaliman--Kutzschebauch formula} (see the proof of Cor.~2.2 in \cite{denscrit}) which can be verified by a straightforward calculation, says that for vector fields $V$ and $W$ and holomorphic functions $f, g, h$ where $h \in \ker W$ we obtain:
\begin{equation}
\label{eqKKformula}
\tag{KK}
[f h V, g W] - [f V, g h W] = - f g V(h) W
\end{equation}
\end{remark}

\begin{remark}
\label{shearcompleteness}
Moreover, we note that if a vector field $V$ is complete and $f \in \ker V$, then $f V$ is complete as well, and if $V(V(h)) = 0$, then $f h V$ is complete too, see e.g.\ Varolin \cite{shears}.
\end{remark}

If we apply \eqref{eqKKformula} to the complete vector fields given by the equations \eqref{eqcompletezerochart1}, \eqref{eqcompletezerochart2} and \eqref{eqcompletezerochart3}, i.e.\
\[
V = z_0^d \frac{\partial}{\partial z_1}, \quad
W = (z_1-\lambda_1)^{d-1} z_0 \frac{\partial}{\partial z_0}, \quad
f = z_0^{j_0}, \quad
g = (z_1 - \lambda_1)^{j_1}
\]
and $h = (z_1 - \lambda_1)$, hence $V(h) = z_0^d$, we obtain that all vector fields of the following form are in the Lie algebra generated by the complete vector fields:
\begin{equation}
\label{eqtrivterms0}
z_0^{j_0} \cdot (z_1 - \lambda_1)^{j_1} \cdot (z_1-\lambda_1)^{d-1} z_0^{d+1} \frac{\partial}{\partial z_0}, \quad j_0, j_1 \in \NN_0
\end{equation}

Moreover, we obtain 
\begin{equation}
\label{eqtrivterms1}
\begin{split}
\left[ z_0^{j_0} \cdot (z_1 - \lambda_1)^{j_1+1} \cdot (z_1-\lambda_1)^{d-1} z_0^{d+1} \frac{\partial}{\partial z_0}, z_0^d \frac{\partial}{\partial z_1} \right] + \\ 
\frac{z_0^d}{j_0+1} z_0^{j_0} \cdot (z_1 - \lambda_1)^{j_1+1} \cdot (z_1-\lambda_1)^{d-1} z_0^{d+1} \frac{\partial}{\partial z_0} \\
 = d z_0^{j_0} \cdot (z_1 - \lambda_1)^{j_1} \cdot (z_1-\lambda_1)^{d-1} z_0^{2d} \frac{\partial}{\partial z_1}, \quad j_0, j_1 \in \NN_0
\end{split}
\end{equation}

Note that the set of functions $\mathrm{span}_\CC \left\{ z_0^{j_0} \cdot (z_1 - \lambda_1)^{j_1}, \; j_0, j_1 \in \NN_0 \right\}$ does not contain any non-trivial ideal, neither in the ring of functions of the coordinate chart, nor in the ring of functions of the surface $S$.


\begin{proposition}
\label{propalgmodule}
Let $S$ be the complex surface given by the equations \eqref{eqdef1} and \eqref{eqdef2}, not necessarily smooth, with polynomials $P_1, \dots, P_d$ which may have multiple roots, and $\lambda_1, \dots, \lambda_{d} \in \CC$ arbitrary. 
Then the following $\regular(S)$-module in the Lie algebra of all algebraic vector fields on $S$ is contained in the Lie subalgebra generated by the complete algebraic vector fields on $S$:
\begin{equation}
\mathrm{span}_\CC \left\{
z_0^{j_0} \cdots z_{d+1}^{j_{d+1}} \cdot R_{d+1} \cdot {\varphi_d}_\ast \left( \frac{\partial}{\partial z_{d+1}} \right)
, \;  j_0, \dots, j_{d+1} \in \NN_0
\right\}
\end{equation}
where $R_{d+1} \in \CC[z_0, \dots, z_{d+1}]$ is a certain polynomial.
\end{proposition}
\begin{proof}
We proceed by induction on $k = 0, 1, \dots, d$ to show that 
\begin{equation}
\mathrm{span}_\CC \left\{
z_0^{j_0} \cdots z_{k+1}^{j_{k+1}} \cdot R_{k+1} \cdot {\varphi_k}_\ast \left( \frac{\partial}{\partial z_{k+1}} \right)
, \;  j_0, \dots, j_{k+1} \in \NN_0
\right\}
\end{equation}
is in the Lie algebra generated generated by the complete algebraic vector fields on $S$.

For $k=0$ we have by \eqref{eqtrivterms1} in the chart $\varphi_0$ that
\[
z_0^{j_0+2d} \cdot (z_1 - \lambda_1)^{j_1+d-1} \frac{\partial}{\partial z_1} = z_0^{j_0} \cdot z_1^{j_1} \cdot R_1 \cdot \frac{\partial}{\partial z_1} , \quad j_0, j_1 \in \NN_0
\]
is in the Lie subalgebra generated by complete algebraic vector fields on $S$ where $R_1 = z_0^{2d} \cdot (z_1 - \lambda_1)^{d-1}$. Note that we can replace $(z_1 - \lambda_1)^{j_1}$ by $z_1^{j_1}$ since we work in the span.

For the induction step $k-1 \mapsto k$ we first note that
\begin{equation}
\label{eqinductionmove}
\varphi_k^\ast {\varphi_{k-1}}_\ast \left( \frac{\partial}{\partial z_{k-1}} \right) = \frac{\partial z_{k+1}}{\partial z_{z_{k-1}}} \cdot \frac{\partial}{\partial z_{k+1}} = -\frac{z_{k+1}}{z_{k-1} - \lambda_{k-1}} \cdot \frac{\partial}{\partial z_{k+1}}
\end{equation}
We used that in the chart $\varphi_{k-1}$ we have $\frac{\partial z_{k}}{\partial z_{k-1}} = 0$ and that the inverse of $\varphi_k$ is just the projection.
We now calculate the following Lie brackets in the chart $\varphi_k$:
\begin{align*}
\begin{split}
\left[ z_0^{j_0} \cdot z_1^{j_1} \cdots z_{k}^{j_{k}} \cdot R_{k}^2 \cdot \frac{-z_{k+1}}{z_{k-1} - \lambda_{k-1}} \cdot \frac{\partial}{\partial z_{k+1}}, \; z_{k+1}^d \cdot (z_k - \lambda_k) \cdot \frac{\partial}{\partial z_k} \right]
\\=
 z_0^{j_0} \cdot z_1^{j_1} \cdots z_{k}^{j_{k}} \cdot R_{k}^2 \cdot \frac{-z_{k+1}}{z_{k-1} - \lambda_{k-1}} \cdot d \cdot z_{k+1}^{d-1} \cdot ( z_k - \lambda_k ) \frac{\partial}{\partial z_{k}}
\\+  \left( \frac{\partial}{\partial z_k}
z_0^{j_0} \cdot z_1^{j_1} \cdots z_{k}^{j_{k}} \cdot \frac{z_{k+1}}{z_{k-1} - \lambda_{k-1}} \right) \cdot R_{k}^2 \cdot \frac{\partial}{\partial z_{k+1}}
\\+ z_0^{j_0} \cdot z_1^{j_1} \cdots z_{k}^{j_{k}} \cdot \frac{z_{k+1}}{z_{k-1} - \lambda_{k-1}} \cdot \frac{\partial R_{k}}{\partial z_k} \cdot 2 R_k \cdot \frac{\partial}{\partial z_{k+1}}
\end{split}
\end{align*}
Note that we interpret the $z_0, \dots, z_{k-1}$ as rational functions of $(z_k, z_{k+1})$.
Rewriting the coordinates $z_{k+1}$'s which occur only up to power $d-1$ in terms of $z_0, \dots, z_{k}$ and raising the powers in $z_0, \dots, z_{k}$ if necessary, we can assume that the $\frac{\partial}{\partial z_{k+1}}$-terms are already in the Lie subalgebra generated by the complete algebraic vector fields due to \eqref{eqinductionmove} and subtract it.
We can continue with the term
\[
z_0^{j_0} \cdot z_1^{j_1} \cdots z_{k}^{j_{k}} \cdot \widehat{R}_k \cdot  \frac{\partial}{\partial z_{k}}
\]
for a suitable polynomial $\widehat{R}_k$.

\begin{equation}
\begin{split}
\left[
z_0^{j_0} \cdot z_1^{j_1} \cdots z_{k}^{j_{k}} \cdot \widehat{R}_k \cdot  \frac{\partial}{\partial z_{k}}, \;
z_{k+1}^{j_{k+1}+d-1} \cdot (z_{k} - \lambda_{k}) \cdot \frac{\partial}{\partial z_{k}} 
\right]
\\= z_0^{j_0} \cdot z_1^{j_1} \cdots z_{k}^{j_{k}} \cdot z_{k+1}^{j_{k+1}} \widehat{R}_{k+1} \frac{\partial}{\partial z_{k}} 
\end{split}
\end{equation}
for a suitable polynomial $\widehat{R}_{k+1}$, again raising the powers if necessary and re-absorbing them into $\widehat{R}_{k+1}$.

To complete the induction, we calculate, again in the chart $\varphi_k$:
\[
\begin{split}
\left[
z_0^{j_0} \cdot z_1^{j_1} \cdots \cdot z_k^{j_k} \cdot z_{k+1}^{j_{k+1}} \cdot \widehat{R}_{k+1}^2 \cdot \frac{\partial}{\partial z_k}, \;
(z_k - \lambda_k)^{d-1} \cdot z_{k+1} \cdot \frac{\partial}{\partial z_{k+1}}\right]
\\=
z_0^{j_0} \cdot z_1^{j_1}  \cdots \cdot z_k^{j_k} \cdot z_{k+1}^{j_{k+1}} \cdot \widehat{R}_{k+1}^2 \cdot (d-1) \cdot (z_k - \lambda_k)^{d-2} \cdot z_{k+1} \cdot \frac{\partial}{\partial z_{k+1}}
\\-(z_k - \lambda_k)^{d-1} \cdot z_{k+1} \cdot \left( \frac{\partial}{\partial z_{k+1}} z_0^{j_0} \cdot z_1^{j_1} \cdots \cdot z_k^{j_k} \cdot z_{k+1}^{j_{k+1}} \cdot \widehat{R}_{k+1}^2 \right) \frac{\partial}{\partial z_k}
\end{split}
\]
Raising the powers in $z_0, \dots, z_{k+1}$ if necessary, we can assume that the last term is already in the Lie subalgebra generated by the complete algebraic vector fields and subtract it. We define $R_{k+1}$ to be $\widehat{R}_{k+1}^2 \cdot (d-1) \cdot (z_k - \lambda_k)^{d-2} \cdot z_{k+1}$ multiplied by these powers.
\end{proof}

\begin{remark}
\label{remccc}
The crucial ingredient to make this proof work is the existence of nice affine charts which form sort of a ``chain'', i.e.\ two of them either intersect in a (punctured) coordinate axis of the ambient space or have empty intersection, hence we call them a \emph{charming chain of charts} that must satisfy the following:
Let $X \subset \CC^{N+1}_{z_0, z_1, \dots, z_N, z_{N+1}}$ be a complex affine-algebraic surface and let $\varphi_0, \varphi_1, \dots, \varphi_N$ be local parametrizations of $X$ with the following properties:
\begin{enumerate}
\item $\varphi_i^{-1}(z_0, \dots, z_{N+1}) = (z_i, z_{i+1})$ for $i = 0, \dots, N$
\item $\mathop{\mathrm{dom}} \varphi_0 \cong \CC \times \CC^\ast$
\item $\mathop{\mathrm{dom}} \varphi_i \cong \CC^\ast \times \CC^\ast$ for $i = 1, \dots, N$
\item\label{cccproppair} There exists a \emph{compatible pair} of complete algebraic vector fields on $\mathop{\mathrm{dom}} \varphi_0$ within $\regular(\CC^2) \subseteq \regular(\mathop{\mathrm{dom}} \varphi_0)$ that extends to $S$.
\item There exist two complete algebraic vector fields $\mu_i, \nu_i$ on each $\mathop{\mathrm{dom}} \varphi_i$ that span the tangent space in every point of $\mathop{\mathrm{dom}} \varphi_i$ and such that they extend to $S$ for $i = 1, \dots, N$.
\end{enumerate}

Note that in Property \ref{cccproppair} above, the ideal of the compatible pair does in general not extend, otherwise the density property would follow easily using the results of Kaliman and Kutzschebauch \cite{denscrit}.
\end{remark}

Let us remind the definition of a compatible pair from \cite{denscrit}:
\begin{definition}
Let $V$ and $W$ be non-trivial algebraic vector fields on an complex affine-algebraic manifold $X$ such that $V$ is a locally nilpotent derivation on $\regular(X)$ and $W$ is either locally nilpotent or semi-simple. We say that $(V, W)$ is a \emph{compatible pair} if
\begin{itemize}
\item $\mathop{\mathrm{span}_\CC} \{ \ker V \cdot \ker W \}$ contains a non-trivial ideal in $\regular(X)$ and
\item $\exists \; h \in \ker W$ with $V(h) \neq 0$ but $V(V(h)) = 0$.
\end{itemize}
\end{definition}

\section{Density Property}
\label{secdensity}

In this section we briefly recall the notions and Theorems we need for proving our result and then prove the Theorems \ref{thmdensprop} and \ref{thmrelativedensprop}.

\begin{definition}[\cite{denscrit}*{Definition 2.2}] \hfill
\begin{enumerate}
\item Let $X$ be an algebraic manifold and $x_0 \in X$.
A finite subset $M$ of the tangent space $T_{x_0} X$ is called a \emph{generating set}
if the image of $M$ under the action of the isotropy subgroup of $x_0$
(in the group of all algebraic automorphisms $\mathrm{Aut}_\mathrm{alg}(X)$)
generates the whole space $T_{x_0} X$.
\item Let $X$ be a complex manifold and $x_0 \in X$.
A finite subset $M$ of the tangent space $T_{x_0} X$ is called a \emph{generating set}
if the image of $M$ under the action of the isotropy subgroup of $x_0$
(in the group of all holomorphic automorphisms $\mathrm{Aut}_\mathrm{hol}(X)$)
generates the whole space $T_{x_0} X$.
\end{enumerate}

\end{definition}

We will make use of the following, central result of \cite{denscrit}:
\begin{theorem}\cite{denscrit}*{Theorem 1}
\label{thmmodulealg}
Let $X$ be an affine algebraic manifold, homogeneous w.r.t.\ $\mathop{\mathrm{Aut}_{\mathrm{alg}}} X$, with algebra of regular functions $\CC[X]$, and let $L$ be a submodule of the $\CC[X]$-module of all algebraic vector fields such that $L \subseteq \mathrm{Lie}_{\mathrm{alg}}(X)$.
Suppose that the fiber of $L$ over some $x_0 \in X$ contains a generating set.
Then $X$ has the algebraic density property.
\end{theorem}

The analogous statement in the holomorphic case follows with essentially the same proof:

\begin{theorem}\cite{lengthfour}*{Theorem 2.3}
\label{thmmodulehol}
Let $X$ be a Stein manifold, \linebreak homogeneous w.r.t.\ $\mathop{\mathrm{Aut}_{\mathrm{hol}}}\! X$\!,
with algebra of holomorphic functions \linebreak $\hol(X)$, and let $L$ be a submodule of the $\hol(X)$-module of all holomorphic vector fields such that $L \subseteq \overline{\mathrm{Lie}_{\mathrm{hol}}(X)}$.
Suppose that the fiber of $L$ over some $x_0 \in X$ contains a generating set.
Then $X$ has the density property.
\end{theorem}

We now apply Theorem \ref{thmmodulehol} to prove the density property of $S$:

\begin{proof}[Proof of Theorem \ref{thmdensprop}]
By Proposition \ref{proptrans} the smooth complex surface $S$ is a $\aut(S)$-homogeneous manifold.
By Proposition \ref{propalgmodule} there exists a non-trivial $\regular(S)$-submodule of $\mathrm{Lie}_{\mathrm{alg}}(S)$, of which we take its closure $L$ with respect to compact convergence. This is a non-trivial $\hol(S)$-submodule in $\overline{\mathrm{Lie}_{\mathrm{hol}}(S)}$.
In order to apply Theorem \ref{thmmodulealg} we need to find a generating set; we follow the proof of \cite{lengthfour}*{Theorem 1.7} and pick a point $p = (p_0, p_1, \dots, p_d, p_{d+1}) \in S$ such that $p_0 \neq 0$ and such that we find a vector field $\mu \in L \neq 0$ which does not vanish in $p$, e.g.\ $\mu =  {\varphi_0}_{\ast}\left( (z_1 - \lambda_1)^{d-1} z_0 \frac{\partial}{\partial z_0} \right)$. The vector field $\nu := {\varphi_0}_\ast\left( z_0^d \frac{\partial}{\partial z_1} \right)$ is complete by Lemma \ref{lemzerochart} and does not vanish in $p$.

The function $f(z_0, \dots, z_{d+1}) := z_0 - p_0$ lies in its kernel and hence by Remark \ref{shearcompleteness}, $f \cdot \nu$ is again complete, and its flow map will fix the point $p$. The induced action of the time-$1$-map on the tangent space $T_{p} S$ is given by $w \mapsto w + d_{p} f \nu(p) = w + {z_0}^d \cdot {\varphi_0}{\ast}\left( \frac{\partial}{\partial z_1} \right)$, see e.g.\ the calculation in \cite{denscrit}*{Corollary 2.8}. Evaluating this expression for $w = \mu(p)$ and noting that $\mu(p)$ is a non-zero vector in direction of ${\varphi_0}_{\ast}\left( \frac{\partial}{\partial z_0} \right)$ we see that $\mu(p)$ is a generating set.
We can now apply Theorem \ref{thmmodulehol} to $S$.
\end{proof}

In the singular or non-transitive algebraic situation, we need the following ingredient from Kutzschebauch--Leuenberger--Liendo:
\begin{theorem}\cite{singular}*{Theorem 2.2}
\label{thmrelativemodule}
Let $X$ be a complex affine-algebraic algebraic variety and let $A$ be a subvariety containing its singularity locus.
Assume that $\autalg(X, A)$, the subgroup of $\autalg(X)$ that stabilizes $A$, acts transitively on $X \setminus A$.
Let $L$ be a finitely generated $\CC[X]$-submodule of $\mathrm{VF}_\mathrm{alg}(X, A)$, the algebraic vector fields on $X$ that vanish on $A$.
Assume that $L$ is contained the Lie algebra generated by the complete vector fields vanishing on $A$. If the fiber of $L$ over some $p \in X \setminus A$ contains a generating set, then $X$ has the algebraic density property relative to $A$.
\end{theorem}

\begin{proof}[Proof of Theorem \ref{thmrelativedensprop}]
The proof is essentially identical to the proof of Theorem \ref{thmrelativemodule}, we just use \ref{thmrelativemodule} instead of Theorem \ref{thmmodulehol}. We observe that the $\CC[S]$-submodule $L$ provided by Proposition \ref{propalgmodule} is indeed finitely generated.
\end{proof}

\section{Geometric Consequences}
\label{secgeom}
We repeat here for comparison the definition of the relative algebraic density property given in the introduction and give the definition of the relative (holomorphic) density property as well:

\begin{definition}\cite{singular}*{Def.~1.1}
Let $X$ be a complex affine-algebraic variety. We say that $X$ has the \emph{algebraic density property relative to a subvariety} $A \subset X$ if $\mathrm{Lie}_{\mathrm{alg}}(X, A)$, the Lie algebra generated by the $\CC$-complete algebraic vector fields on $X$ that vanish in $A$ and $\mathrm{VF}_{\mathrm{alg}}(X, A)$, the Lie algebra of all algebraic vector fields on $X$ that vanish on $A$ satisfy
\[
\mathcal{I}_A^\ell \cdot \mathrm{VF}_{\mathrm{alg}}(X, A) \subseteq \mathrm{Lie}_{\mathrm{alg}}(X, A)
\]
for some $\ell \in \NN_0$.
\end{definition}

\begin{definition}\cite{singular}*{Def.~6.1}
Let $X$ be a Stein space. We say that $X$ has the \emph{density property relative to an analytic subvariety} $A \subset X$ if $\mathrm{Lie}_{\mathrm{hol}}(X, A)$, the Lie algebra generated by the $\CC$-complete holomorphic vector fields on $X$ that vanish in $A$ and $\mathrm{VF}_{\mathrm{hol}}(X, A)$, the Lie algebra of all holomorphic vector fields on $X$ that vanish on $A$ satisfy
\[
\mathcal{I}_A^\ell \cdot \mathrm{VF}_{\mathrm{hol}}(X, A) \subseteq \overline{\mathrm{Lie}_{\mathrm{hol}}(X, A)}
\]
for some $\ell \in \NN_0$.
\end{definition}

It is easy to see that the (relative) algebraic density property implies the (relative) density property, see \cite{singular}*{Prop.~6.2}.

The main motivation for establishing the (relative) density property of the (relative) density property is the so-called \emph{Anders\'en--Lempert Theorem} which is a Runge-type approximation theorem for holomorphic automorphisms.

\begin{theorem}\cite{singular}*{Theorem 6.3}
Let $X$ be a normal reduced Stein space and let $A \subset X$ be a closed analytic subvariety that contains the singularity locus of $X$. Let $\Omega \subseteq X$ be an open subset and $\varphi \colon [0,1] \times \Omega \to X$ be a $\cont^1$-smooth map such that
\begin{enumerate}
\item $\varphi_0 \colon \Omega \to X$ is the natural embedding,
\item $\varphi_t \colon \Omega \to X$ is holomorphic and injective for every $t \in [0,1]$,
\item $\varphi_t(\Omega)$ is a Runge subset of $X$ for every $t \in [0,1]$, and
\item $\varphi_t$ fixes $A$ up to order $\ell$ where $\ell$ is such that \\
$\mathcal{I}_A^\ell \cdot \mathrm{VF}_{\mathrm{hol}}(X, A) \subseteq \overline{\mathrm{Lie}_{\mathrm{hol}}(X, A)}$.
\end{enumerate}
Then for every $\varepsilon > 0$ and for every compact $K \subset \Omega$ there exists a continuous family $\Phi \colon [0, 1] \to \aut(X)$, fixing $A$ pointwise, such that
$\Phi_0 = \id_X$ and $\| \varphi_t - \Phi_t \|_K < \varepsilon$ for all $t \in [0,1]$.

\smallskip
Moreover, these automorphisms can be chosen to be compositions of flows from a dense Lie subalgebra in $\mathrm{Lie}_{\mathrm{hol}}(X, A)$, see Varolin \cite{Varolin1}.
\end{theorem}

For the case of star-shaped domain $\Omega \subset \CC^n = X$ and $A = \emptyset$ the Theorem is due to Anders\'en and Lempert \cite{AndersenLempert}. For a Runge domain $\Omega \subset \CC^n = X$ and $A = \emptyset$ it is due to Forstneri\v{c} and Rosay \cites{ForstnericRosay, ForstnericRosayErr}. The general version for a Stein manifold with the density property, but with $A = \emptyset$ was stated by Varolin \cites{Varolin1, Varolin2}.

\smallskip
A consequence of this Theorem is again holomorphic flexibility and $m$-transitivity. Both properties we have already established as a consequence of Proposition \ref{proptrans} for these Gizatullin surfaces that are treated in our article.

The Theorem also enables us to describe the $\cont^1$-path-identity component of the group of holomorphic automorphisms of $X$. To prove the transitivity, the existence of the submodule and for the generating set we only needed the $\CC$-complete vector fields given in Section \ref{seccharts}. Therefore we obtain the following description:

\begin{theorem}
\label{thmautos}
The following automorphisms of the surface $S$ described by the equations \eqref{eqdef1} and \eqref{eqdef2} generate a dense subgroup of the $\cont^1$-path-identity component of the group of holomorphic automorphisms of $S$:
\begin{align*}
(z_1, z_2) &\mapsto (z_0, z_1 + t z_0^{j+d}) \\
(z_1, z_2) &\mapsto (z_0, \exp(t z_0^{j+d-1} ) \cdot (z_1 - \lambda_1) + \lambda_1) \\
(z_1, z_2) &\mapsto (\exp(t (z_1-\lambda_1)^{j+d-1}) \cdot z_0, z_1) \\
(z_i, z_{i+1}) &\mapsto (\exp(t z_{i+1}^{j+d-1} ) \cdot (z_i - \lambda_i) + \lambda_i, z_{i+1}) \\
(z_i, z_{i+1}) &\mapsto (z_{i}, \exp(t z_{i}^{j+d-i-1} ) \cdot (z_{i+1} - \lambda_{i+1}) + \lambda_{i+1})
\end{align*}
where $j \in \NN_0$ and $t \in \CC$ and $i=1, \dots, d$. The formulas are to be understood in the charts given by $\varphi_0, \dots, \varphi_d$.
\end{theorem}

\begin{definition}
Let $X$ be a complex manifold and let $\Omega \subsetneq X$ be a proper subdomain.
\begin{enumerate}
\item We call $\Omega$ a \emph{Fatou--Bieberbach domain of the first kind} if there exist a biholomorphic map $\CC^n \to \Omega$, $n = \dim_\CC X$.
\item We call $\Omega$ a \emph{Fatou--Bieberbach domain of the second kind} if there exist a biholomorphic map $X \to \Omega$.
\end{enumerate}
\end{definition}

The density property implies the existence of both kinds of Fatou--Bieberbach domains according to Varolin \cite{Varolin2}. Since the surfaces $S$ defined by the equations  \eqref{eqdef1} and \eqref{eqdef2} admit at most one singular point and an at most finite subvariety $A$ of fixed points, Varolin's contruction also works if $S$ has only the relative density property.

\bigskip
By a result of the Wold and the author \cite{embedRiemann} every open Riemann surface can be properly holomorphically immersed in any smooth Stein surface with density property. Thus, the Gizatullin surfaces treated in this article are also potential targets for properly holomorphically embedding Riemann surfaces that a priori may not embed into $\CC^2$.

\section{Acknowledgement}
The author would like to thank Adrien Dubouloz for helpful comments and discussions.

\end{document}